\documentclass{ifacconf}
\usepackage{graphicx}      
\usepackage{svg}
\usepackage{natbib}        
\usepackage{amsmath,amssymb}
\usepackage{xcolor}
\usepackage{algpseudocode}
\usepackage{mathtools}
\usepackage{tikz}
\usepackage{float}
\newtheorem{theorem}{Theorem}
\newtheorem{proposition}{Proposition}
\newtheorem{remark}{Remark}

\definecolor{coolblue}{RGB}{0,114,189}

\DeclareMathOperator{\Ann}{Ann}
\DeclareMathOperator{\rank}{rank}

\usepackage{mathtools}
\usepackage{microtype}

\newenvironment{proof}{\paragraph*{Proof:}}{\hfill$\square$}
\begin{document}
\begin{frontmatter}

\title{Optimal Control of Hybrid Systems with Submersive Resets\thanksref{footnoteinfo}} 

\thanks[footnoteinfo]{This work was funded by AFOSR Award No. MURI FA9550-23-1-0400.}

\author[First]{William A. Clark} 
\author[Second]{Maria Oprea}
\author[Third]{Aden Shaw}

\address[First]{Department of Mathematics, Ohio University, Athens, OH 45701 USA (e-mail: clarkw3@ohio.edu)}
\address[Second]{Center for Applied Mathematics, Cornell University, Ithaca, NY 14853 USA (e-mail: mao237@cornell.edu)}
\address[Third]{Department of Mathematics, Rose-Hulman Institute of Technology, Terre Haute, IN 47803 USA, (e-mail: shawap@rose-hulman.edu)}
\begin{abstract}                
Hybrid dynamical systems are systems which posses both continuous and discrete transitions. Assuming that the discrete transitions (resets) occur a finite number of times, the optimal control problem can be solved by gluing together the optimal arcs from the underlying continuous problem via the ``Hamilton jump conditions.'' In most cases, it is assumed that the reset is a diffeomorphism (onto its image) and the corresponding Hamilton jump condition admits a unique solution. However, in many applications, the reset results in a drop in dimension and the corresponding Hamilton jump condition admits zero/infinitely many solutions. A geometric interpretation of this issue is explored in the case where the reset is a submersion (onto its image). Optimality conditions are presented for this type of reset along with an accompanying numerical example.
\end{abstract}


\end{frontmatter}

\section{Introduction}
Hybrid dynamical systems contain both continuous and discrete evolution and are used to study a wide-range of phenomena, \cite{HDS_magazine}. A specific class of hybrid systems are those whose discrete transition is triggered by an event in the continuous dynamics. Such dynamics will be described via
\begin{equation}\label{eq:sHDS}
    \mathcal{H} : 
    \begin{cases}
        \dot{x} = f(x), & x \not\in \mathcal{S}, \\
        x^+ = \Delta(x^-), & x \in \mathcal{S}.
    \end{cases}
\end{equation}
Here, $f \colon M \to TM$ is a smooth vector-field which encodes the continuous evolution, $\mathcal{S} \subset M$ is an embedded, codimension 1, submanifold which dictates where the discrete events occur and is called the \textit{guard}, and $\Delta \colon \mathcal{S} \to M$ is the discrete event called the \textit{reset}. A trajectory following \eqref{eq:sHDS} is called a hybrid arc. A schematic of these dynamics where the reset fails to be injective is illustrated in Fig. \ref{fig:Dynamics Schematic}.
\begin{figure}[!ht]
    \centering
    \includegraphics[width = .95\linewidth]{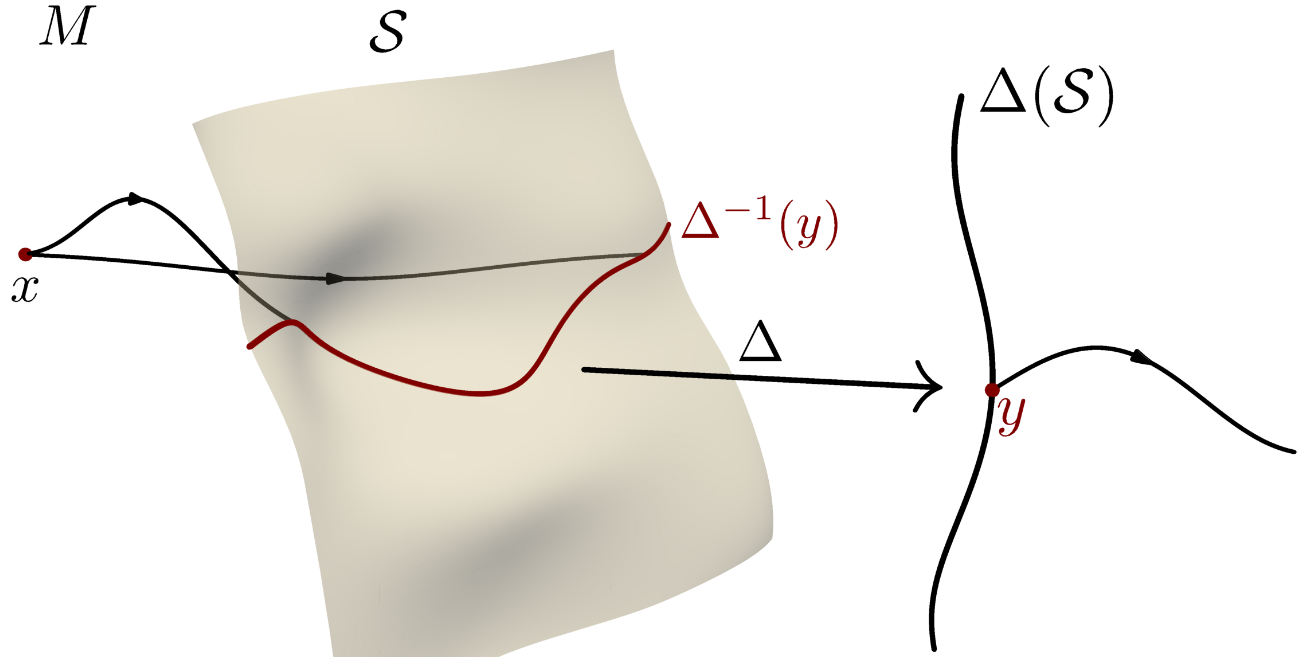}
    \caption{Schematic of the dynamics described by \eqref{eq:sHDS}, where $M$ is given by the ambient space, $S$ is a $2$ dimensional surface, and its image under $\Delta$ is a curve. Two different controlled trajectories emanating from the same initial point may be mapped to the same state at the reset.}
    \label{fig:Dynamics Schematic}
\end{figure}

Controls can be implemented in the system \eqref{eq:sHDS} in three principal ways: 
\begin{enumerate}
    \item through the continuous dynamics, $f \colon M \times \mathcal{U} \to TM$,
    \item through the event location, $\mathcal{S}_u \subset M$ for $u \in \mathcal{U}$,
    \item through the reset map, $\Delta \colon \mathcal{S} \times \mathcal{U} \to M$.
\end{enumerate}
For this work, we will only consider hybrid systems controlled through their continuous dynamics; introducing controls through the reset have been studied in, e.g. \cite{pakniyat_2023}. Hybrid control systems will henceforth have the form
\begin{equation}\label{eq:cHDS}
    \mathcal{HC} : 
    \begin{cases}
        \dot{x} = f(x, u), & x \not\in \mathcal{S}, \\
        x^+ = \Delta(x^-), & x \in \mathcal{S},
    \end{cases}
\end{equation}
where $f \colon M\times \mathcal{U} \to TM$ is the controlled vector-field which is assumed to be smooth in both variables.

The problem that we consider is the minimization of
\begin{equation}\label{eq:minimize}
    \mathcal{J}(u; t_0, x_0) 
    = \int_{t_0}^T \, \ell(x(s), u(s))\, ds + \varphi(x(T)),
\end{equation}
where $\ell \colon M \times \mathcal{U} \to \mathbb{R}$ and $\varphi \colon N \subset M \to \mathbb{R}$ are smooth, and subject dynamics are given by \eqref{eq:cHDS} with the conditions $x(t_0) = x_0$ and $x(T) \in N$ for some target manifold.

Necessary conditions for a hybrid arc with finitely many resets to be a weak minimum of \eqref{eq:minimize} are given by the hybrid maximum principle (HMP), \cite{Böhme2017,DMITRUK2008964,pakniyat_2023}. If $H \colon T^*M \to \mathbb{R}$ is the optimal Hamiltonian prescribed by the HMP, then at resets the co-states jump according to:
\begin{equation}\label{eq:costate_jump}
    \begin{split}
        p^+ \circ \Delta_* - p^- & \in \Ann(T\mathcal{S}), \\
        H^+ - H^- & = 0,
    \end{split}
\end{equation}
where $\Delta_* \colon T\mathcal{S} \to TM$ is the push-forward and $\Ann(T\mathcal{S})$ is the annihilator.
Unfortunately, it is not known a priori that \eqref{eq:costate_jump} admits a unique solution for $p^+$. In the case where $\Delta_*$ has full rank (the reset is a local diffeomorphism onto its image), the first condition in \eqref{eq:costate_jump} always admits solutions (although the second condition may not).

In the case where $\Delta$ causes a drop in dimension, i.e. $\dim(\Delta(\mathcal{S})) < \dim(\mathcal{S})$, its push-forward, $\Delta_*$, no longer has full rank. Resets of this type are not rare; an example of such a case is locomotion with plastic impacts, \cite{locomotion,passive_walking}. When this happens, \eqref{eq:costate_jump} generally admits zero or infinitely many solutions as the first condition in \eqref{eq:costate_jump} becomes over-determined. The contribution of this work is to present a geometric solution on how to resolve this issue. This culminates in Algorithm \ref{alg:alg1} which offers a way to systematically find a correct value for the co-state jumps.

\section{Optimal Control of Hybrid Systems}
This work is concerned with hybrid control systems of the form \eqref{eq:cHDS}. To be explicit on the assumptions of this system, they are
described by the tuple $\mathcal{HC} = (M, \mathcal{U}, f, \mathcal{S}, \Delta)$ where
\begin{enumerate}
    \item $M$ is a finite-dimensional manifold,
    \item $\mathcal{U} \subset \mathbb{R}^m$ is a closed subset,
    \item $f \colon M \times \mathcal{U} \to TM$ is a smooth controlled vector field,
    \item $\mathcal{S}\subset M$ is an embedded codimension 1 submanifold, and
    \item $\Delta \colon \mathcal{S} \to M$ is a smooth map such that $\Delta(\mathcal{S})$ is embedded and the map is a submersion onto its image with constant rank $r$.
\end{enumerate}
The set $\mathcal{U}$ is called the control set and a piece-wise smooth curve $u \colon \mathbb{R}^+ \to \mathcal{U}$ is an admissible control. The resulting evolution for this control is specified by \eqref{eq:cHDS}.

\begin{remark}
    The key ingredient for the optimal control problem at resets is not the guard, rather the fibers of $\Delta$. This is the reason for requiring $\Delta$ to have constant rank. By the submersion theorem, cf. Theorem 5.12 in \cite{lee}, each fiber $\Delta^{-1}(y)$ is a properly embedded submanifold of codimension $r$ in $\mathcal{S}$, equivalently codimension $r+1$ in $M$.
\end{remark}

For a given initial condition, $x_0 \in M$, the optimal control problem is to minimize the objective \eqref{eq:minimize} over all admissible controls such that the controlled trajectory satisfies the terminal constraint, $x(T)\in N$.
Necessary conditions for a solution to this control problem are given by the hybrid Pontryagin maximum principle is which fundamentally a concatenation of the classical maximum principle. We let the control Hamiltonian be (assuming the extremal is regular) 
\begin{gather*}
    \mathcal{H} \colon T^*M \times \mathcal{U} \to \mathbb{R}, \\
    \mathcal{H}(x, p; u) \coloneqq \langle p, f(x, u)\rangle + \ell(x, u),
\end{gather*}
where $\langle\cdot,\cdot\rangle$ denotes the pairing between covectors and vectors. The optimized Hamiltonian is the function
\begin{gather*}
    H \colon T^*M \to \mathbb{R} ,\quad
    H(x, p) = \min_{u \in \mathcal{U}} \, \mathcal{H}(x, p; u).
\end{gather*}
The classical maximum principle states that if a controlled trajectory $(\bar{x},\bar{u})$ is optimal, then there exists a lifted curve $(\bar{x},\bar{p}) \colon [t_0,T]\to T^*M$ such that 
\begin{equation*}
    \dot{\bar{x}} = \frac{\partial H}{\partial p}, \quad \dot{\bar{p}} = -\frac{\partial H}{\partial x},
\end{equation*}
subject to the transversality condition
\begin{equation*}
    \bar{x}(T) \in N, \quad i^*\bar{p}(T) = d\varphi_{\bar{x}(T)},
\end{equation*}
where $i \colon N \hookrightarrow M$ is the inclusion.

For the hybrid optimal control problem, an optimal trajectory satisfies the same Hamiltonian evolution between resets and the same transversality condition as the classical version. However, at resets, the state jumps via $\Delta$ and the co-state jumps by a ``Hamiltonian jump condition,'' $\tilde{\Delta} \colon T^*M|_\mathcal{S} \to T^*M$, cf. \cite{pakniyat_2023}. This is made precise in the following theorem.

\begin{theorem}[Hybrid maximum principle]\label{thm:HMP}
    Let $(\bar{x},\bar{u})$ be a controlled hybrid trajectory over the interval $[t_0,T]$ such that there exists finitely many resets and the reset times are uniformly separated. If this trajectory is optimal, then there exists a lifted hybrid curve $(\bar{x},\bar{p}):[t_0,T]\to T^*M$ such that away from resets, the curve follows the Hamiltonian flow
    \begin{equation}\label{eq:hmp_continuous}
        \dot{\bar{x}} = \frac{\partial H}{\partial p}, \quad \dot{\bar{p}} = -\frac{\partial H}{\partial x}, \quad \bar{x}\not\in\mathcal{S}.
    \end{equation}
    While at resets, the curve jumps according to
    \begin{equation}\label{eq:hmp_discrete}
        \left. \begin{aligned}
            \bar{x}^+ &= \Delta(\bar{x}^-) \\
            \bar{p}^+\circ\Delta_* - \bar{p}^- &\in \Ann(TS) \\
            H^+ - H^- &= 0
        \end{aligned}\right\}
        \quad \bar{x}^- \in \mathcal{S}.
    \end{equation}
    Finally, the curve is subject to the terminal conditions
    \begin{equation}\label{eq:hmp_terminal}
        \bar{x}(T) \in N, \quad i^*\bar{p}(T) = d\varphi_{\bar{x}(T)},
    \end{equation}
    where $i:N\hookrightarrow M$ is the inclusion.
\end{theorem}
The assumptions that the number of resets be finite and that the resets are separated are vital as there exist hybrid systems whose optimal trajectory is Zeno and does not satisfy the HMP, cf. \cite{zeno_hOC}.
\section{Systems with Submersive Resets}
A general technique to determine optimal trajectories is to integrate the hybrid trajectory, \eqref{eq:hmp_continuous} and \eqref{eq:hmp_discrete}, and determine the initial co-state, $p(t_0) \in T_{x_0}^*M$, such that the terminal condition \eqref{eq:hmp_terminal} is satisfied. This approach is challenging from both the numerical and theoretical standpoints. Numerically, it requires solving a highly discontinuous boundary value problem, and traditional techniques like the shooting method fall short, cf. \cite{hyrbidshooting}. Theoretically, it is not guaranteed that \eqref{eq:hmp_discrete} admits a unique solution. If a finite number of solutions exist, a search can be performed. However, in the case where $\Delta$ has rank $r < \dim\mathcal{S}$, $\Delta_*$ is a degenerate map and it is expected that either zero or infinitely many solutions should exist. 
We focus on the consistency conditions for the over-determined system \eqref{eq:hmp_discrete} to assist with existence.

Consider the linear system of equations $Ax = b$ where $x \in \mathbb{R}^n$ is the unknown, $b \in \mathbb{R}^m$, and $A$ is an $m \times n$ matrix with $m<n$. There exist no solutions for $x$ if $b\not\in \mathrm{image}(A)$ and infinitely many if $b \in \mathrm{image}(A)$. As solvability of this system implicitly constrains the vector $b$, solvability of \eqref{eq:hmp_discrete} implicitly constrains $p^-$ - the co-state immediately before reset. These constraints are
\begin{equation}\label{eq:consistency}
    p^-(v) = 0, \quad \forall v \in T\mathcal{S} : \Delta_*v = 0.
\end{equation}
A geometric interpretation of these consistency constraints is  given in the following proposition.
\begin{proposition}\label{prop:1}
    Suppose that $\bar{x} \colon [t_0,T]\to M$ is an optimal  trajectory, and for some $t_0 < t^* < T$, $\bar{x}(t^*)\in\mathcal{S}$. Then $y \coloneqq \bar{x}|_{[t_0,t^*]}$ is a minimum to
    \begin{equation*}
        \int_{t_0}^{t^*} \, \ell\left( y(s), u(s) \right) \, ds,
    \end{equation*}
    subject to the boundary conditions
    \begin{equation*}
        y(t_0) = \bar{x}(t_0), \quad y(t^*)\in N_{\bar{x}}\coloneqq\Delta^{-1}\left( \Delta(\bar{x}(t^*))\right).
    \end{equation*}
\end{proposition}
\begin{proof}
    As this entire manifold, $N_{\bar{x}}$, gets mapped to the same point under $\Delta$, the location of impact on this manifold does not change the resulting trajectory.
\end{proof}
Intuitively, this proposition states that once the optimal trajectory is known, each individual continuous arc is optimal with known initial and terminal conditions determined by the optimal path. 
The (hybrid) maximum principle applied to this restricted control problem states that the co-state has the terminal condition
\begin{equation}\label{eq:momentum_impact}
    i^*p(t^*) = 0, \quad i \colon N_{\bar{x}} \hookrightarrow M.
\end{equation}
This condition is identical to the consistency condition \eqref{eq:consistency}.

When solving the general optimal control problem, the reset locations are not known a priori. The condition \eqref{eq:momentum_impact} at an unspecified point $x(t^*) \in \mathcal{S}$ can be formulated as
\begin{equation*}
    p(t^*) \in \mathrm{Ann}(T\mathcal{F}),
\end{equation*}
where $\mathcal{F}$ is $\mathcal{S}$ decomposed into the fibers of $\Delta$, i.e.
\begin{equation*}
    \mathcal{F} = \bigsqcup_{s \in \Delta(\mathcal{S})} \, \Delta^{-1}(s), \quad
    T\mathcal{F} = \left\{v \in T\mathcal{S} : \Delta_*v = 0 \right\},
\end{equation*} 
as shown in Fig. \ref{fig:Foliation Schematic}.
\begin{figure}
    \centering
    \includegraphics{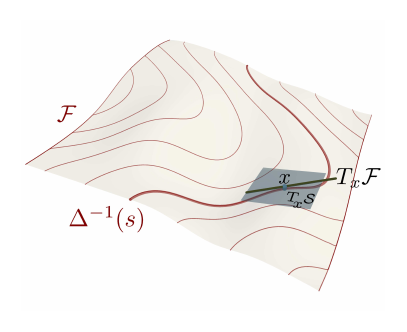}
    \caption{The foliation, $\mathcal{F}$, of the guard by the reset along with $T_x\mathcal{F}\subset T_x\mathcal{S}$.}
    \label{fig:Foliation Schematic}
\end{figure}
These consistency conditions modify the hybrid maximum principle in the following way.

\begin{theorem}\label{thm:2}
    Let $(x(t),p(t))$ be an optimal trajectory described in Theorem \ref{thm:HMP} with distinct reset times $\{t_k\}_{k=1}^N$. At these times,
    \begin{equation}\label{eq:consistency2}
        p^-(t_k)\in\Ann(T\mathcal{F}),
    \end{equation}
    and
    \begin{equation}\label{eq:subm_jump}
        p^+(t_k) 
        \in p^-(t_k) \circ \Delta_*^\dagger + \Ann\left(T\Delta(\mathcal{S})\right),
    \end{equation}
    subject to the constraint $H^+ = H^-$ where $\Delta_*^\dagger$ is a right pesudo-inverse of $\Delta_*$. 
\end{theorem}
\begin{proof}
    Equation \eqref{eq:consistency2} follows from Proposition \ref{prop:1}. Equation \eqref{eq:subm_jump} comes from the definition of $\Delta_*^\dagger$ along with the fact that $\forall p \in \Ann(TS) \implies p\circ \Delta_*^\dagger \in \Ann(T\Delta(S))$.
\end{proof}
Although the consistency conditions allow for solutions to \eqref{eq:subm_jump} to generally exist, such solutions are far from unique. Let the set-valued map for all possible solutions to \eqref{eq:subm_jump} be
\begin{equation*}
    \tilde{\Delta} \colon \mathrm{Ann}(T\mathcal{F}) \rightrightarrows T^*M.
\end{equation*}
Assuming sufficient regularity, the set of all possible solutions forms a manifold of known dimension.
\begin{proposition}\label{prop:dim_jump}
    Let $(x,p)\in\Ann(T\mathcal{F})$ and $E:= H(x,p)$. Consider the restricted Hamiltonian function,
    \begin{align*}
        h 
        \colon & \Ann\left( T_{\Delta(x)}\Delta(\mathcal{S})\right) \to\mathbb{R}, \\
        &  \mu \mapsto H\big(\Delta(x), p\circ\Delta_*^\dagger + \mu\big).
    \end{align*}
    If $E$ is a regular value of $h$, then $\tilde{\Delta}(x,p)$ is a manifold of dimension
    \begin{equation*}
        \dim \tilde{\Delta}(x,p) = \dim M - \rank\Delta - 1.
    \end{equation*}
\end{proposition}
\begin{proof}
    his follows directly from the implicit function theorem.
\end{proof}

\subsection{Uniqueness of the Co-state Jump}
In the case where $\Delta$ does not have full rank, the dimension of the solution manifold is greater than 0; hence solutions to the Hamiltonian jump condition are not unique. There exist infinitely many possible choices $p^+ \in \tilde{\Delta}(x^-, p^-)$. We now focus on how to single out a solution.

Let $\varphi_t\colon T^*M\to T^*M$ be the Hamiltonian flow from \eqref{eq:hmp_continuous}.
As the consistency conditions only exist at resets, consider the return map to $\mathcal{S}$, 
\begin{gather*}
    \tau \colon T^*M\to\mathbb{R}, \\
    \tau(x,p) = \min \left\{ t \geq 0 : \pi_M\left( \varphi_{t}(x,p)\right) \in \mathcal{S}\right\},
\end{gather*}
where $\pi_M:T^*M\to M$ is the canonical projection and $\tau(x,p) = \infty$ if no return exists.
For each $x\in M$, let $\Xi_x \subset T^*_xM$ be the co-states that satisfy the consistency conditions at the next reset as depicted in Fig. \ref{fig:Xi sets}, i.e.
\begin{equation*}
    \Xi_x = \left\{p \in T_x^*M : \varphi_{\tau(x,p)}(x,p) \in \mathrm{Ann}(T\mathcal{F})\right\}.
\end{equation*}
Although it is not clear how these sets look like, we do know that if $x\in \mathcal{S}$, then $\tau = 0 $ so  $\Xi_x = \Ann(T_x\mathcal{F})$ and has dimension $\rank\Delta +1$.

\begin{figure}%
    \centering
    \includegraphics[width = \linewidth]{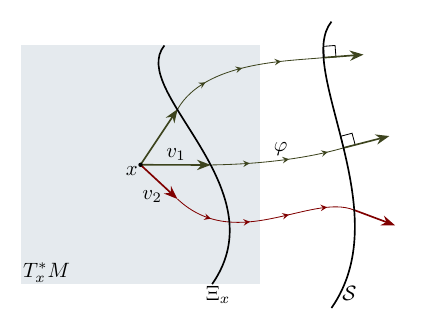}
    \caption{Depiction of $\Xi_x$ for $x \in M$, where $\varphi_{\tau(x,v_1)}(x,v_1)$ is perpendicular to $T_y\mathcal{S}$ for some $y$ whereas $\varphi_{\tau(x,v_2)}(x,v_2)$ is not.}
    \label{fig:Xi sets}
\end{figure}
Suppose that an optimal trajectory has a reset at the state $(x,p)$. Then the co-state post-reset must lie within both $\tilde{\Delta}(x,p)$ and $\Xi_{\Delta(x)}$, i.e.
\begin{equation*}
    p^+ \in \tilde{\Delta}(x,p) \cap \Xi_{\Delta(x)}.
\end{equation*}
Generically, this intersection contains a finite number of points, since $\Xi_{\Delta(x)}$ and $\Tilde{\Delta}(x, p)$ have complimentary dimension (due to Proposition \ref{prop:dim_jump} above)
\begin{equation*}
    \begin{split}
        \dim \tilde{\Delta}(x,p) & + \dim\Xi_{\Delta(x)} \\
        & = \dim M - \rank\Delta - 1 + \rank\Delta + 1 \\
        & = \dim M.
    \end{split}
\end{equation*}
The procedure to select the correct reset is shown in greater detail in Algorithm 1.  

\begin{algorithm}\label{alg:alg1}
    \caption{Procedure to select the correct reset at time $t$ assuming that $x(t) \in S$. Takes in the Hamiltonian $H$, the current state and costate $x(t)$ and $p(t)$, the reset map $\Delta$ and the flow $\varphi_t$. It outputs the co-state after impact $p^+$. $\pi_{T^*M}$ denotes the projection onto the cotangent component, while $\pi_M$ denotes the projection onto the manifold.  }
    \begin{algorithmic}
        \State $x_0, p_0 \gets x(t), p(t)$
        \State  $D_1 \gets \Delta_*(x_0, p_0) $
        \State $D_1^\dagger \gets D_1^T(D_1^TD_1)^{-1}$
        \Function{Reset}{$\mu$}
            \State output(1) = $H(p(t)) - H(p(t)D_1^\dagger + \mu)$ 
            \State \textbf{Solve} $\pi_M(\varphi_{t^*}(\Delta(x_0), p_0D^\dagger_1 + \mu)) \in  S$ for $t^* >0 $
            \State $p(t^*) \gets \pi_{T^*M}\left(\varphi_{t^*}(x_0, p_0)\right)$
            \State $D_2 = \Delta_*(x(t^*), p(t^*))$
            \State output(2) = $D_2^T (D_2^TD_2)^{-1}D_2p(t^*) - p(t^*)$
            \State output(3) = $D_1^\dagger D_1\mu- \mu$
        \EndFunction
        \State \textbf{Solve}: RESET$(\mu) = 0 $ 
        \State $p^+ = p_0D^\dagger_1 + \mu$
    \end{algorithmic}
\end{algorithm}
\vspace{-0.2cm}
The resulting optimal trajectory follows the hybrid dynamics: 
\begin{equation*}
    \begin{cases}
        \dot{x} = \frac{\partial H}{\partial p}, \quad \dot{p} = -\frac{\partial H}{\partial x}, & x\not\in \mathcal{S} \\[2ex]
        \left. \begin{aligned}
            x^+ &= \Delta(x^-) \\
            p^+ &\in \tilde{\Delta}(x^-,p^-)\cap \Xi_{x^+}
        \end{aligned}\right\}
        & x^- \in \mathcal{S},
    \end{cases}
\end{equation*}
subject to the boundary conditions
\begin{align*}
    x(0) = x_0, & \qquad p(0) \in \Xi_{x_0}, \\
    x(T) \in N, & \quad i^*p(T) = d\varphi_{x(T)}.
\end{align*}

\begin{figure*}[!ht]
    \centering
    \includegraphics[width = 0.9\textwidth]{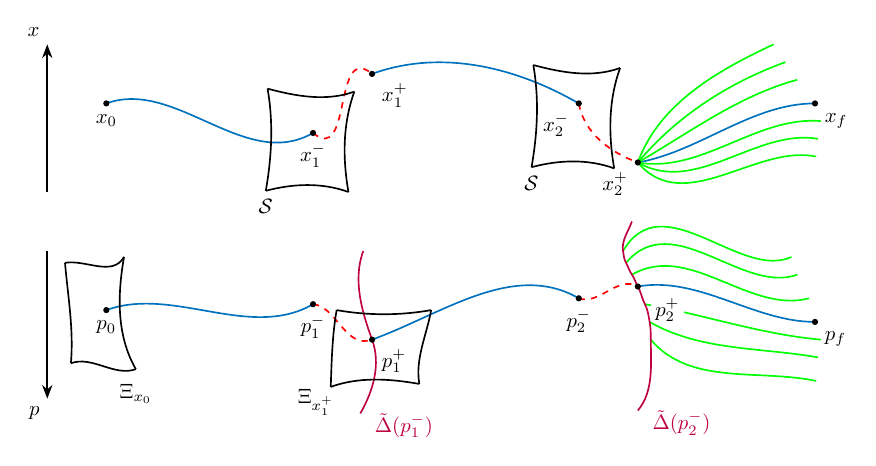}
    \caption{
    A schematic for the optimal control problem with two resets. The initial co-state must be admissible, $p_0\in \Xi_{x_0}$. Following the first reset, the new co-state must be contained in $\Xi_{x_1^+}\cap \tilde{\Delta}(p_1^-)$. After the second (and final) reset, the co-state no longer needs to be admissible. Rather, it needs to be chosen such that the terminal condition is satisfied. }
    \label{fig:two_bounces}
\end{figure*}

Unfortunately, while there are $n$ terminal conditions that need satisfied, there are only $\rank\Delta+1$ choices in the initial conditions. To overcome this apparent over-determined system, notice that the consistency conditions are not needed during the final reset. As such, during the last jump, we are free to choose any $p^+\in\tilde{\Delta}(x^-,p^-)$ which resolves this over-determination. A schematic of this procedure with two resets is depicted in Fig. \ref{fig:two_bounces}.

\section{Bouncing Ball with an ``Internal Variable''}
To demonstrate the above analysis, we study the example of a controlled bouncing ball with an ``internal state.'' The ball's position and velocity are $x$ and $y$, respectively while the internal variable is $z$.
This variable is independent of the motion of the ball, and can be directly controlled.  Whenever an impact happens, the internal variable is set to a constant value of $z^+ = 1$. In this case, the reset map has rank $2<3$. 
The control problem is to find
\begin{equation*}
    \min_{u,v} \, \frac{1}{2}\int_0^5 \, \left(u^2 + v^2 + z^2\right) \, dt,
\end{equation*}
subject to the continuous dynamics
\begin{equation*}
    \left. \begin{array}{l}
    \dot{x} = y \\ \dot{y} = -1 + u \\ \dot{z} = v
    \end{array} \right\} \quad x>0,
\end{equation*}
along with the reset
\begin{equation*}
    \Delta(x,y,z) = (x, -y, 1), \quad \begin{array}{c} x = 0, \\ y < 0, \end{array}
\end{equation*}
and the terminal conditions
\begin{equation*}
    x(5) = 1, \quad y(5) = 0, \quad z(5) = 0.
\end{equation*}

The control and optimized Hamiltonians are
\begin{equation*}
    \mathcal{H} = \frac{1}{2}\left(u^2+v^2+z^2\right) + p_xy + p_y(u-1) + p_zv,
\end{equation*}
\begin{equation*}
    H = -\frac{1}{2}p_y^2 - \frac{1}{2}p_z^2 + \frac{1}{2}z^2 + p_xy - p_y,
\end{equation*}
where the optimal controls are
\begin{equation*}
        u^* = -p_y, \quad
        v^* = -p_z.
\end{equation*}
The resulting equations of motion are given by
\begin{equation}\label{eq:full_system}
    \begin{array}{ll}
        \dot{x} = y, & \dot{p}_x = 0,\\
        \dot{y} = -1 - p_y, \qquad & \dot{p}_y = -p_x,\\
        \dot{z} = -p_z, & \dot{p}_z = -z.\\
    \end{array}
\end{equation}
The condition $p^+\circ\Delta_* - p^-\in\Ann(T\mathcal{S})$ translates to
\begin{equation*}
    \begin{bmatrix}
        p_x^+ & p_y^+ & p_z^+ 
    \end{bmatrix}\begin{bmatrix}
        1 & 0 & 0 \\ 0 & -1 & 0 \\ 0 & 0 & 0
    \end{bmatrix} = \begin{bmatrix}
        p_x^- & p_y^- & p_z^-
    \end{bmatrix} + \begin{bmatrix}
        \varepsilon & 0 & 0
    \end{bmatrix}.
\end{equation*}
In order for the above system to be solvable, we need $p_z^-=0$ (as predicted in Theorem \ref{thm:2}). Writing this out, we have
\begin{equation*}
    \begin{split}
        p_x^+ &= p_x^- + \varepsilon \\
        p_y^+ &= -p_y^- \\
        p_z^+ &= A
    \end{split}
\end{equation*}
where the unknowns $A$ and $\varepsilon$ are related through the energy conservation condition, which yields 
\begin{equation*}
    \begin{split}
        H^- &= -\frac{1}{2}(p_y^-)^2 + \frac{1}{2}z^2 + p_x^-y^- - p_y^- \\
        &= H^+ \\
        &= -\frac{1}{2}(p_y^-)^2 - \frac{1}{2}A^2 + \frac{1}{2} - (p_x^- + \varepsilon)y^- + p_y^-
    \end{split}
\end{equation*}
Solving for $\varepsilon$ yields:
\begin{equation*}
    p_x^- + \varepsilon = \frac{1}{y}\left( -\frac{1}{2}z^2 + \frac{1}{2} -p_x^-y^- + 2p_y^- - \frac{1}{2}A^2 \right)
\end{equation*}
Hence there are infinitely many lifted reset maps given by
\begin{equation}\label{eq:full_reset}
    \begin{split}
        p_x^+ &= \frac{1}{y}\left[ \frac{1}{2}\left( 1 - z^2 - A^2\right) -p_x^-y^- + 2p_y^-\right] \\
        p_y^+ &= -p_y^- \\
        p_z^+ &= A.
    \end{split}
\end{equation}
In this example, $\tilde{\Delta}$ is the curve parameterized by all possible values of $A$ in \eqref{eq:full_reset}.
In order to pinpoint which of these solutions is correct, we apply the consistency conditions \eqref{eq:consistency} at the \textit{next} reset, i.e. requiring that $p_z=0$ at the next instance when $x=0$.
\subsection{Computing the Admissible Co-States}
In this example, the sets $\Xi_{(x,y,z)}$ can be explicitly computed. For simplicity, assume that the current time is $t=t_0$ and the states are $z(t_0) = z_0$ and $p_z(t_0) = A$. The evolution of these two states are given by
\begin{equation*}
    \begin{split}
        z(t) &= \left(\frac{z_0-A}{2}\right)e^{\tau} + \left( \frac{z_0+A}{2}\right) e^{-\tau}, \\
        p_z(t) &= \left(\frac{A-z_0}{2}\right)e^\tau + \left(\frac{A+z_0}{2}\right)e^{-\tau},
    \end{split}
\end{equation*}
where $\tau = t - t_0$ is the time elapsed. 
The condition $p_z(t^*)=0$ can be solved to yield
\begin{equation}\label{eq:admissible_impact_time}
    t^* = \frac{1}{2}\ln\left(\frac{z_0+A}{z_0-A}\right) + t_0, \quad 0\leq A < z_0.
\end{equation}
Inverting this produces
\begin{equation*}
    A = \left( \frac{e^{2(t^*-t_0)}-1}{e^{2(t^*-t_0)}+1}\right) z_0.
\end{equation*}
The set of admissible co-states are then given by
\begin{equation*}
    \Xi_{(x_0, y_0, z_0)} = \left\{ 
    (p_x, p_y, p_z) : x\left( t^*\right) = 0
    \right\},
\end{equation*}
where $x(\,\cdot\,)$ is the $x$-trajectory of \eqref{eq:full_system} and $t^*$ is given by \eqref{eq:admissible_impact_time}. This set is a surface in $T_{(x_0,y_0,z_0)}\mathbb{R}^3 \cong \mathbb{R}^3$.

Fortunately, the entire continuous dynamics \eqref{eq:full_system} can be explicitly integrated and the condition $x(t^*)=0$ can be solved in closed-form. 
For an arbitrary initial state, we have:
\begin{equation*}
    \Xi_{(x,y,z)} : \begin{cases}
        \tau = \frac{1}{2}\ln\left( \frac{1+p_z}{1-p_z}\right), \\
        p_x = \frac{3}{\tau}(1+p_y) - \frac{6}{\tau^3}x - \frac{6}{\tau^2}y.
    \end{cases}
\end{equation*}
At the moment of a reset, $(0,y,z)\in\mathcal{S}$, the admissible co-states post-reset are given by
\begin{equation*}
    \Xi_{(0,-y,1)} : \begin{cases}
        \tau = \frac{1}{2}\ln\left(\frac{1+p_z}{1-p_z}\right), \\
        p_y = \left(\frac{1}{3}\tau\right)p_x - \left( \frac{2y}{\tau} + 1\right).
    \end{cases}
\end{equation*}

To find the correct jump condition at a reset, we need to determine the intersection between $\tilde\Delta(x,y,z;p_x,p_y,p_z)$ and $\Xi_{(0, -y, 1)}$.
The curve $\tilde\Delta$ is only over the range $0\leq A < 1$ as $z=1$ immediately post-reset. A figure of this intersection is shown in Fig. \ref{fig:find_A} where there does, indeed, exist a unique intersection.

\begin{figure}
    \centering
    \includegraphics[width=\columnwidth]{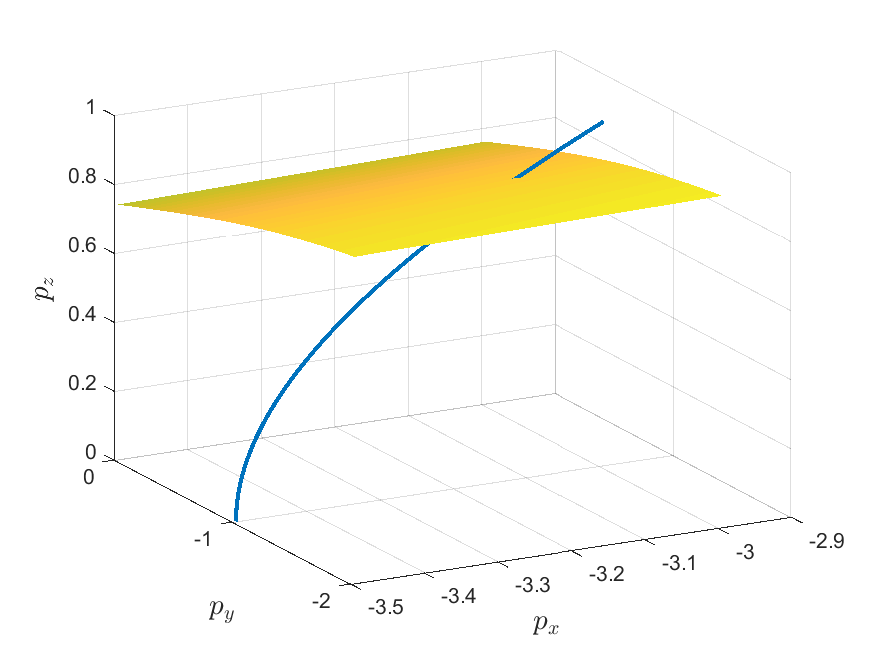}
    \caption{A plot of $\Xi_{(0,1,1)}$ (the yellow surface) along with $\tilde{\Delta}(0,-1,0.1,1,1,0)$ (the blue curve). In this example, there exists a single intersection point.}
    \label{fig:find_A}
\end{figure}

The intersection point in Fig. \ref{fig:find_A}, the lifted reset map at this point is
\begin{equation*}
    \left[
    \begin{array}{l}
        x^- = 0 \\
        y^- = -1 \\
        z^- = 0.1 \\
        p_x^- = 1 \\
        p_y^- = 1 \\
        p_z^- = 0
    \end{array}\right] \mapsto \left[ \begin{array}{c}
        0 \\ 1 \\ 1 \\ -3.1050 \\ -1.0000 \\ 0.8832
    \end{array}\right]
\end{equation*}
Using the initial condition $(x_0, y_0, z_0) = (0.5, 0, 1)$, optimal trajectories corresponding to 0, 1, 2, and 3 resets are depicted in Fig. \ref{fig:opt_traj} while the costs for each of these trajectories are displayed in Table \ref{tab:my_label}. Over these four arcs, the minimizing trajectory has two resets.
\begin{figure}
    \centering
    \includegraphics[width=\columnwidth]{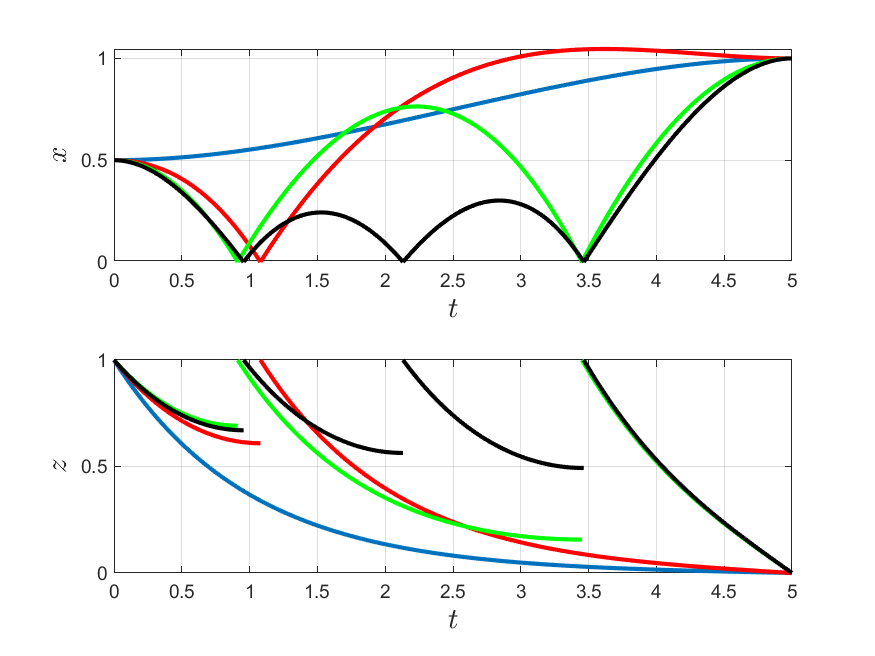}
    \caption{The optimal trajectories for 0, 1, 2, and 3 bounces (blue, red, green, and black respectively).}
    \label{fig:opt_traj}
\end{figure}
\begin{table}
    \centering
    \begin{tabular}{c|c}
       \# Bounces  &  Cost \\ \hline
        0 & 3.012 \\
        1 & 2.0809 \\
        2 & 1.4685 \\
        3 & 2.2376
    \end{tabular}
    \caption{The costs of the four trajectories in Fig. \ref{fig:opt_traj}.}
    \label{tab:my_label}
\end{table}
\section{Conclusion}
This work is just the first step towards dealing with the question of existence and uniqueness of solutions to the hybrid adjoint equation \eqref{eq:hmp_discrete}. We provide sufficient conditions to guarantee the existence of solutions, in the case when the reset map is a submersion. In particular, if  the consistency condition \eqref{eq:consistency} is satisfied, then there exists infinitely many solutions. In order to pinpoint the correct one we propagate the manifold of possible solutions forward, and select the ones that guarantee that \eqref{eq:consistency} is satisfied at the next impact time. This procedure produces the desired results, as exemplified in the case of a bouncing ball with an internal variable. However, in order to numerically compute this, we have to solve a hybrid boundary value problem, with an additional root finding problem at each impact time. While this might be doable if the dimension is low, if we increase the dimension, the complexity of the algorithm increases exponentially. As such, an optimization of this procedure, and an extension to higher dimensions is needed in the future. 

From a theoretical perspective, the reset map is not limited to being a submersion. There are hybrid systems where $\Delta$ does not have constant rank. Consider for example a bouncing ball where, at every impact the velocity gets cubed. The impact map $\Delta(v) = -v^3$ has a singularity at $v = 0$. We would like to extend this analysis to accommodate for singularities.

Lastly, the core of the analysis presented here lies in the treatment of a loss in dimension due to the reset map $\Delta$. We show that the problem is still solvable even though there is an apparent loss of information due to the non-injectivity of $\Delta$ and the existence of infinitely many solutions. In the future, we would like to transport the insights gained from this work to the case when this loss comes from the higher codimension of the impact surface instead.

\bibliography{ref.bib}             

\end{document}